\documentclass{amsart}
\usepackage{amssymb,latexsym}
\usepackage{amsmath}
\usepackage{graphicx}
\usepackage{textcomp}

\usepackage{amsthm,amssymb,enumerate,graphicx, tikz}
\usepackage{amscd}
\usepackage{setspace}
\usepackage{comment}
\usepackage{hyperref}
\usepackage{cleveref}

\def\dd{\hbox{-}}

\title[]{Graphs with no even holes and no sector wheels are the union of two chordal graphs}

\author{Tara Abrishami} \thanks{T. Abrishami:
Department of Mathematics, University of Hamburg, Germany.\\ \url{tara.abrishami@uni-hamburg.de}. (This work was performed while the author was at Princeton University.) Supported by NSF-EPSRC Grant DMS-2120644 and by AFOSR grant FA9550-22-1-0083}
\author{Eli Berger} \thanks{
E. Berger: Department of Mathematics, University of Haifa,  Israel. \url{berger@math.haifa.ac.il}.}
\author{Maria Chudnovsky} \thanks{M. Chudnovsky: Department of Mathematics, Princeton University, USA.\\ \url{mchudnov@math.princeton.edu}. Supported by NSF-EPSRC Grant DMS-2120644 and by AFOSR grant FA9550-22-1-008.}
\author{Shira Zerbib} \thanks{S. Zerbib: Department of of Mathematics, Iowa State University, USA. \url{zerbib@iastate.edu}. Supported by NSF Grant DMS-1953929. \\E. Berger, M. Chudnovsky, and S. Zerbib were also supported by BSF grant 2016077.}

\newtheorem{thm}{Theorem}[section]

\newtheorem{lem}[thm]{Lemma}

\usepackage{enumerate}
\newcounter{tbox}

\newcommand{\sta}[1]{\vspace*{0.3cm}\refstepcounter{tbox}\noindent{\parbox{\textwidth}{(\thetbox) \emph{#1}}}\vspace*{0.3cm}}

\begin{document}

\begin{abstract}
Sivaraman \cite{conjecture} conjectured that if $G$ is a graph with no induced even cycle then there exist sets $X_1, X_2 \subseteq V(G)$ satisfying $V(G) = X_1 \cup X_2$ such that the induced graphs $G[X_1]$ and $G[X_2]$ are both chordal. We prove this conjecture in the special case where $G$ contains no sector wheel, namely, a pair $(H, w)$ where $H$ is an induced cycle of  $G$ and $w$ is a vertex in $V(G) \setminus V(H)$ such that $N(w) \cap H$ is either $V(H)$ or a path with at least three vertices. 
\end{abstract}
\maketitle
\section{Introduction}

The degree $d(v)$ of a vertex $v$ in a graph $G$  is the number of edges in $G$ containing $v$.
The length of a path or a cycle is the number of edges in it. 
A graph is {\em even-hole-free} if it does not contain an induced cycle of even length. Even-hole-free graphs are the subject of intense interest and study in structural graph theory. Much is known about the structure of even-hole-free graphs: for example, even-hole-free graphs have a decomposition theorem \cite{daSilva2013Decomposition2-joins}, are known to have {\em bisimplicial vertices} (vertices whose neighborhoods are the union of two cliques) \cite{ehf-bisimplicial}, and can be recognized in polynomial time \cite{recognition}. Even-hole-free graphs have also been well-studied with respect to algorithmic parameters such as independence number, chromatic number, and treewidth. See \cite{ehf-survey} for a survey of even-hole-free graphs. 

A graph is {\em chordal} if it contains no induced cycles of length four or greater. In 2020, Sivaraman conjectured that every even-hole-free graph can be written as the union of two chordal induced subgraphs \cite{conjecture}. In this paper, we prove Sivaraman's conjecture under an additional assumption. 

A {\em wheel} $(H, w)$ is a hole $H$ and a vertex $w \in V(G) \setminus H$ such that $|N(w) \cap H| \geq 3$. A {\em universal wheel} is a wheel $(H, w)$ such that $w$ is complete to $H$. A {\em sector wheel} is a wheel $(H, w)$ such that either $(H, w)$ is a universal wheel or $N(w) \cap H$ is a path. A pair $(X_1, X_2)$ of vertex subsets of $G$ such that $X_1 \cup X_2 = V(G)$ and $G[X_1]$ and $G[X_2]$ are chordal is called a {\em chordal cover of $G$}.  
 We prove: 
\begin{thm}\label{thm:overall-thm}
    Every even-hole-free graph with no sector wheel admits a chordal cover. 
\end{thm}

\subsection{Proof outline}
Our proof depends on a decomposition theorem for even-hole-free graphs proved in \cite{daSilva2013Decomposition2-joins}. We first need some definitions. Let $T$ be a tree and write $V_1, V_2$ for its two sides when viewed as a bipartite graph. Let $L$ denote the set of leaves of $T$ and write $L_1 = L \cap V_1$, $L_2 = L \cap V_2$.
For each $v \in L$ we write $e(v)$ for the unique edge of $T$ incident with $v$. 
We construct a graph $B(T)$ as follows: the set of vertices of $B(T)$ is $E(T) \cup \{x_1,x_2\}$, where $x_1, x_2$ are two additional vertices. Two vertices of $B(T)$ are adjacent if one of the following holds:

\begin{itemize}
    \item They represent two edges of $T$ with a common vertex, or
    \item one of them is $x_i$ and the other is $e(v)$ and $v \in L_i$ for some $i \in \{1,2\}$,  or
    \item they are $x_1$ and $x_2$.
\end{itemize}

Note that the vertex set of every induced cycle in $B(T)$ with at least 4 vertices consists of the edge set of some path in $T$ between two leaves together with either $x_1$ or $x_2$ or both. A graph $G$ is an {\em extended nontrivial basic graph} if $G = B(T)$ for some tree $T$ with at least three leaves and at least two non-leaves. (Note that if $T$ is a path graph then $B(T)$ is a cycle, and if $T$ is a star then $B(T)$ is a clique. Hence it makes sense to exclude these cases and deal with them separately.)

A {\em 2-join} of a graph $G$ is a partition $(A_1, C_1, B_1, A_2, C_2, B_2)$ of $V(G)$ such that the following hold:
\begin{itemize}
    \item  $A_1$ is complete to $A_2$,  $B_1$ is complete to $B_2$,  and there are no other edges of $E(G)$ with one end in $Z_1 := A_1 \cup C_1 \cup B_1$ and one end in $Z_2 := A_2 \cup C_2 \cup B_2$, and 
    \item for $i = 1, 2$, $Z_i$ contains an induced path $M_{i} = (a, m_1, \ldots, m_k, b)$ with one end $a\in A_i$, one end $b\in B_i$, and $\{m_1, \hdots, m_k\} \subseteq C_{i}$ (where $k$ may be $0$), which we call the {\em marker path} for $Z_i$, and $Z_i$ is not just this path.
\end{itemize}

A {\em pyramid} is a graph consisting of a vertex $a$ called the {\em apex}, a triangle $b_1b_2b_3$ called the {\em base}, and three paths $P_i$ from $a$ to $b_i$, each of which has length at least one, at most one of which has length exactly one, such that the only edge from $P_i \setminus \{a\}$ to $P_j \setminus \{a\}$ is $b_ib_j$ for all $\{i, j\} \subseteq \{1, 2, 3\}$. A graph $G$ has a {\em star cutset} if $G$ is connected and if there is a vertex $v \in V(G)$ and a set $C \subseteq N[v]$ with $v \in C$ such that $G \setminus C$ is not connected. The set $C$ is called a {\em star cutset} of $G$. A {\em clique cutset} of a graph $G$ is a set $C \subseteq V(G)$ such that $C$ is a clique and $G \setminus C$ is not connected. A star cutset is {\em proper} if it is not a clique cutset.

We now can state the decomposition theorem we use:

\begin{thm}[\cite{daSilva2013Decomposition2-joins}]
\label{thm:decomposition}
Let $G$ be an even-hole-free graph. Then one of the following holds: 
\begin{itemize}
    \item $G$ is a clique; 
    \item $G$ is a hole;
    \item $G$ is a pyramid; 
    \item $G$ is an extended nontrivial basic graph;
    \item $G$ has a 2-join; or
    \item $G$ has a star cutset. 
\end{itemize}
\end{thm}

Let $G$ be an even-hole-free graph with no sector wheel. The main idea of our proof is to start with a ``precover'' of $G$, i.e. two sets $W_1, W_2 \subseteq V(G)$ such that $G[W_1]$ and $G[W_2]$ are chordal, and extend the precover to a chordal cover of $G$ by finding sets $X_1, X_2 \subseteq V(G)$ such that $W_1 \subseteq X_1$, $W_2 \subseteq X_2$, $X_1 \cup X_2 = V(G)$, and $G[X_1]$ and $G[X_2]$ are chordal. We define the ``precover'' using flat paths in $G$.

For a path $P = (v_1, \ldots, v_k)$, we write $N[P]$ for the set of vertices either in the path or with at least one neighbor in the path. We define the {\em interior} of $P$ to be $int(P) =\{ v_2, \ldots, v_{k-1}\}$. For $k \in \{1,2\}$ we set $int(P) = \emptyset$. 
We say that an induced path $P$
is {\em flat} if all the vertices in its interior have degree 2 in $G$. 
Note that every path with either 1 or 2 vertices is flat.

 We say that a graph $G$ is {\em flat path extendable} (FPE) if for every 
induced 
flat path $P$ and every two sets $W_1, W_2$ such that $G[W_1], G[W_2]$ are chordal, $W_1 \cap W_2 =V(P)$, and  $W_1 \cup W_2 = N[P]$, there exist sets $X_1 \supseteq W_1$ and $X_2 \supseteq W_2$ such that $G[X_1], G[X_2]$ are chordal, $X_1 \cap X_2 =V(P)$, and $X_1 \cup X_2 = V$. Under these conditions, $(P, W_1, W_2)$ is called a {\em precover}, and $(X_1, X_2)$ is a chordal cover of $G$ that {\em extends} $(P, W_1, W_2)$. 

If $G$ is not FPE but every proper induced subgraph of $G$ is FPE, then we say that $G$ is {\em minimal non flat path extendable} (MNFPE), and for a path $P$ not satisfying the above property (i.e., there exist two sets $W_1, W_2$, such that $G[W_1], G[W_2]$ are chordal and $W_1 \cap W_2 =V(P)$ and $W_1 \cup W_2 = N[P]$, but there do not exist sets $X_1 \supseteq W_1$ and $X_2 \supseteq W_2$, such that $G[X_1], G[X_2]$ are chordal and $X_1 \cap X_2 =V(P)$ and $X_1 \cup X_2 = V$) we say that $P$ is a {\em witness path} for $G$ and that $W_1,W_2$ are the corresponding {\em  witness sets}.

We prove the following theorem:

 \begin{thm}
\label{thm:main-nosc}
Every graph with no even hole, no sector wheel, and no star cutset is FPE.
\end{thm}

To deal with the case when $G$ contains a star cutset, we define a closely related concept called weakly flat path extendable. A graph $G$ is {\em weakly flat path extendable} (weakly FPE) if for every path $P$ of length zero or one, and every two sets $W_1, W_2$ such that $G[W_1], G[W_2]$ are chordal, $W_1 \cap W_2 = V(P)$, and $W_1 \cup W_2 = N[P]$,  there exist sets $X_1 \supseteq W_1$ and $X_2 \supseteq W_2$ such that $G[X_1], G[X_2]$ are chordal, $X_1 \cap X_2 = V(P)$, and $X_1 \cup X_2 = V(G)$. Under these conditions, $(P, W_1, W_2)$ is a {\em precover} and $(X_1, X_2)$ is a chordal cover of $G$ that {\em extends} $(P, W_1, W_2)$. (The only difference between weakly flat path extendable and flat path extendable is that weakly flat path extendable only considers paths of length at most one). We note the following relationships between FPE and weakly FPE: 
\begin{itemize}
    \item If $G$ is FPE, then $G$ is weakly FPE. 
    \item If $G$ is minimal non-weakly FPE, then $G$ is not FPE, but $G$ is also not necessarily MNFPE. 
\end{itemize}

We prove: 

\begin{thm}
\label{thm:main}
Every graph with no even hole and no sector wheel is weakly FPE.
\end{thm}
Theorem \ref{thm:main-nosc} (and the stronger definition of FPE) is needed to prove Theorem \ref{thm:main} in the case when $G$ does not contain a proper star cutset.

Theorem \ref{thm:main} implies Theorem \ref{thm:overall-thm}:

\begin{proof}[Proof of Theorem \ref{thm:overall-thm}]
Let $G$ be an even-hole-free graph with no sector wheel. Let $v \in V(G)$. Since $G$ has no sector wheel, it follows that $N[v]$ is chordal. Now, $(\{v\}, N[v], \{v\})$ is a precover of $G$. By Theorem \ref{thm:main}, $G$ is weakly FPE, so $G$ admits a chordal cover. This completes the proof. 
\end{proof}

Theorem \ref{thm:main} is not true without the assumption that the graph has no sector wheel. Indeed, consider the graph $G$ depicted in Figure \ref{fig:ex}. Let $P = \{x\}$ and let $W_1=\{x,y_1,y_3,y_5\}$ and $W_2=\{x,y_2,y_4,y_6\}$. Then there are no $X_1, X_2$   such that $X_1 \cup X_2 = V(G)$,  $W_1 \subseteq X_1$, $W_2 \subseteq X_2$, and  $G[X_1], G[X_2]$ are chordal. Therefore the method in this paper cannot be extended to the case of graphs containing sector wheels without some new ideas.

\begin{figure}[ht]
    \centering

    \begin{tikzpicture}[scale=.5]
    \coordinate (a) at (-6,0);
    \coordinate (b) at (-4,0);
    \coordinate (c) at (-1,0);
    \coordinate (d) at (1,0);
    \coordinate (e) at (4,0);
    \coordinate (f) at (6,0);
    \coordinate (g) at (0,3);
    \coordinate (h) at (0,-2.5);
    \coordinate (i) at (-1.5,-4);
    \coordinate (j) at (1.5,-4);

     \fill[black, draw=black, thick] (a) circle (3pt) node[black, left ] {$y_1$};
    \fill[black, draw=black, thick] (b) circle (3pt) node[black, right ] {$y_2$};
    \fill[black, draw=black, thick] (c) circle (3pt) node[black, left ] {$y_3$};
    \fill[black, draw=black, thick] (d) circle (3pt) node[black, right] {$y_4$};
    \fill[black, draw=black, thick] (e) circle (3pt) node[black, left] {$y_5$};
        \fill[black, draw=black, thick] (f)  circle (3pt) node[black, right] {$y_6$};
            \fill[black, draw=black, thick] (g) circle (3pt) node[black, above] {$x$};
                \fill[black, draw=black, thick] (h) circle (3pt) node[black, below] {$z_2$};
        \fill[black, draw=black, thick] (i)  circle (3pt) node[black, below] {$z_1$};
            \fill[black, draw=black, thick] (j) circle (3pt) node[black, below] {$z_3$};

 \draw (g) -- (a);
    \draw (g) -- (b);
    \draw (g) -- (c);
    \draw (g) -- (d);
    \draw (g) -- (e);
    \draw (g) -- (f);
    \draw (a) -- (b);
    \draw (c) -- (d);
    \draw (e) -- (f);
    \draw (a) -- (i);
    \draw (b) -- (i);
    \draw (c) -- (h);
    \draw (d) -- (h);    
    \draw (e) -- (j);
    \draw (f) -- (j); 
    \draw (i) -- (j);
    \draw (i) -- (h);       
    \draw (j) -- (h);      
    \end{tikzpicture}
\label{fig:ex}
\caption{A graph with no even hole which is not FPE.}
\end{figure}
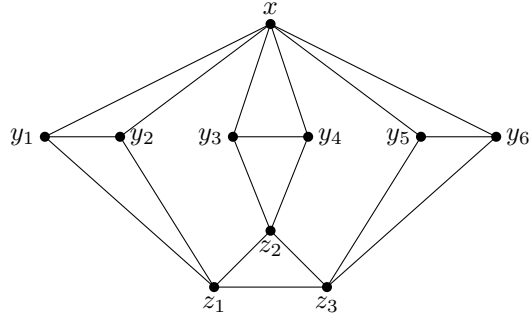

\subsection{Organization of the paper}

In Section \ref{sec:basic}, we prove that if $G$ is an extended nontrivial basic graph, then $G$ is FPE. In Section \ref{sec:2-join}, we prove that if $G$ has no clique cutset and no star cutset, then $G$ is FPE. In Section \ref{sec:clique}, we prove that if $G$ admits a clique cutset, then $G$ is weakly FPE.  In Section \ref{sec:star-cutset}, we prove that if $G$ admits a proper star cutset, then $G$ is weakly FPE. Finally, in Section \ref{sec:final}, we prove Theorem \ref{thm:main}.

\section{Basic graphs} 
\label{sec:basic}
In this section, we prove that extended nontrivial basic graphs with no star cutsets are FPE.

A vertex $v$ in a graph $G$ is  {\em nearly simplicial} if $N(v)$ is the union of a clique and a singleton. First we show that every nearly simplicial vertex in an MNFPE graph $G$ is contained in the neighborhood of every witness path for $G$.  

\begin{lem}
\label{cliqueplussingleton}
    Let $G$ be MNFPE and let $P$ be a witness path for $G$. Then all nearly simplicial vertices of $G$ are in $N[P]$. 
\end{lem}
\begin{proof}
Let $W_1$ and $W_2$ be the witness sets for $P$. Suppose there exists a nearly simplicial vertex $u \in V(G) \setminus N[P]$. Let $G' = G \setminus \{u\}$. Since $G$ is MNFPE, it follows that $G'$ is FPE and $N[P] \subseteq G'$. Let $(X_1, X_2)$ be a chordal cover of $G'$ that extends $(P, W_1, W_2)$. Assume $N(u) = C \cup \{u'\}$ where $C$ is a clique. Let $i \in \{1, 2\}$ such that $u' \in X_i$. Let $X_i' = X_i$ and $X_{3-i}' = X_{3-i} \cup \{u\}$. Now, $(X_1', X_2')$ is a chordal cover of $G$ that extends $(P, W_1, W_2)$, contradicting that $G$ is MNFPE. 
\end{proof}

\begin{lem}
Let $G = B(T)$ be an extended nontrivial basic graph for some tree $T$. If there are two leaves of $T$ with a common neighbor, then $B(T)$ has
a star cutset.
\end{lem}

\begin{proof}
    Let $t_1$ and $t_2$ be two leaves of $T$ with a common neighbor $u$, and let $\ell_1$ and $\ell_2$ be the vertices of $B(T)$ corresponding to the edges $\{u,t_1\}$ and $\{u,t_2\}$, respectively. Up to symmetry between $x_1$ and $x_2$, assume that $x_1$ is adjacent to $\ell_1$. Since $t_1$ and $t_2$ have distance $2$, which is an even number, it follows that $x_1$ is adjacent to $\ell_2$, and that $\{\ell_1, \ell_2\}$ is anticomplete to $x_2$. Then, the neighborhood of $\ell_i$ for $i = 1, 2$ consists of $x_1$ and the vertices of $B(T)$ corresponding to exactly the edges incident with the common neighbor of $t_1$ and $t_2$. In particular, $N_{B(T)}[\ell_1] = N_{B(T)}[\ell_2]$. 


Now, $N_{B(T)}[\ell_1] \setminus \{\ell_2\}$ is a star cutset that separates $\ell_2$ from $B(T) \setminus N_{B(T)}[\ell_1]$. 
\end{proof}


Since we deal with graphs with clique cutsets and proper star cutset separately in Sections \ref{sec:clique} and \ref{sec:star-cutset}, respectively, we may assume here that there are no two leaves in $T$ with a common neighbor.
   
\begin{lem}
\label{lem:basic}
Let $T$ be a tree, let $G = B(T)$ be an extended nontrivial basic graph, and assume that $B(T)$ has no star cutset. 
    Then $B(T)$ is not MNFPE.
\end{lem}
\begin{proof}
We assume for contradiction that there exists a witness path $P$ for $B(T)$.
If $T$ is a path then $B(T)$ is a cycle and therefore is clearly FPE. It follows that at least one of $x_i$ has degree at least 3, so not both  $x_1,x_2$ are internal in $P$.

Let $(t_1, t_2, \ldots, t_k)$ be the longest path in $T$. Assuming $T$ is not a path and there are no two leaves of $T$ with a common neighbor, we must have $k \geq 5$ and $d(t_2) = d(t_{k-1}) = 2$. Note that if $x_i$ is internal in $P$ then $x_{3-i}$ is in $P$.
Indeed, suppose $x_1$ is internal in $P$. Then, by definition of a witness path, $x_1$ has degree 2, and its two neighbors are also in $P$, implying $x_2 \in P$.

Suppose $k=5$. Then $T$ is a subdivision of a star, with $t_3$ being its center. 
In this case, all the edges of $T$ are nearly simplicial in $B(T)$, and therefore, by Lemma \ref{cliqueplussingleton}, all of them are in $N[P]$, as vertices in $B(T)$.
Moreover, $P$ must contain at least one of the vertices $x_1$ and $x_2$,
for otherwise, there is some nearly simplicial vertex not in $N[P]$, contradicting Lemma \ref{cliqueplussingleton}.
Thus $N[P] = V(B(T))$. 
This is impossible, since we have $W_1\cup W_2 = N[P]$. 
So from now on we assume $k>5$.

Since no two leaves of $T$ have a common neighbor, by Lemma \ref{cliqueplussingleton}, we  have that $\{t_1, t_2\}$, $\{t_2, t_3\}$, $\{t_{k-2}, t_{k-1}\}$, and  $\{t_{k-1}, t_k\}$ are all in $N[P]$ since they are nearly simplicial.
Since $\{t_2, t_3\} \in N[P]$, some edge of $T$ that is incident with either $t_2$ or $t_3$ must be in $V(P)$. Similarly, since      $\{t_{k-2}, t_{k-1}\} \in N[P]$, some edge of $T$ that is incident with  either $t_{k-2}$ or $t_{k-1}$ must be in $V(P)$. Since  not both  $x_1,x_2$ are internal in $P$,
this implies, by induction, that $\{t_3,t_4\}, \ldots, \{t_{k-3},t_{k-2}\} \in V(P)$.
Since $\{t_1, t_2\} \in N[P]$, we must have 
$\{t_2, t_3\} \in V(P)$ and similarly, since $\{t_{k-1}, t_k\} \in N[P]$, we must have 
$\{t_{k-2}, t_{k-1}\} \in V(P)$. This implies
$\{t_3,t_4\}, \ldots, \{t_{k-3},t_{k-2}\} \in int(P)$ and hence $d(t_3) = \ldots = d(t_{k-3}) = 2$. We conclude that $T$ is a path, which yields a contradiction as discussed above.
\end{proof}

\section{Graphs with no star cutset}
\label{sec:2-join}

In this section, we prove that every (even hole, sector wheel)-free graph with no star cutset is FPE. We first focus on the case when $G$ admits a 2-join. 
If $G$ admits a 2-join $(A_1, C_1, B_1, A_2, C_2, B_2)$, we denote by  
  $B(Z_i)$  the graph formed  by adding to $Z_i=A_i \cup C_i \cup B_i$ the  marker path $M_{3-i}$.
We call $B(Z_1)$ and $B(Z_2)$ the {\em blocks of decomposition}
of the 2-join $(A_1, C_1, B_1, A_2, C_2, B_2)$. We need the following theorem from \cite{daSilva2013Decomposition2-joins}: 
\begin{thm}[\cite{daSilva2013Decomposition2-joins}, Theorem 2.10]
\label{thm:block-no-sc}
    If $G$ is even-hole-free and has no star cutset, then $B(Z_1)$ and $B(Z_2)$ have no star cutset. 
\end{thm}

The next lemma states how paths and holes interact with the structure of a 2-join. 
\begin{lem}
\label{lem:small-intersection}
   Let $G$ be a graph and let $(A_1, C_1, B_1, A_2, C_2, B_2)$ be a 2-join of $G$. Let $Q$ be a path or a hole of $G$. If $Q \cap A_1 \neq \emptyset$ and $Q \cap A_2 \neq \emptyset$, then $|Q \cap (A_1 \cup A_2)| \leq 3$. Similarly, if $Q \cap B_1 \neq \emptyset$ and $Q \cap B_2 \neq \emptyset$, then $|Q \cap (B_1 \cup B_2)| \leq 3$. 
\end{lem}
\begin{proof}
Suppose first that $|Q \cap A_1| \geq 2$ and $|Q \cap A_2| \geq 2$. Then, $Q \cap (A_1 \cup A_2)$ contains $C_4$ as a subgraph, so $Q$ is not a path. 
Then, either $Q \cap (A_1 \cup A_2)$ is an induced $C_4$, contradicting the fact that $G$ has no even holes, or  $Q \cap (A_1 \cup A_2)$ contains $K_4$ minus an edge as a subgraph, contradicting the fact that $Q$ is a path or a hole.
Therefore, up to symmetry  we may assume that $|Q \cap A_1| = 1$. Suppose $|Q \cap A_2| \geq 3$. Then, $Q \cap (A_1 \cup A_2)$ contains $K_{1, 3}$ as a subgraph, contradicting the fact that $Q$ is a path or a hole. 
\end{proof}

Lemma \ref{lem:small-intersection} has the following useful corollary. 

\begin{lem}
    Let $G$ be a graph and let $(A_1, C_1, B_1, A_2, C_2, B_2)$ be a 2-join of $G$. Let $P$ be a flat path of $G$. If $P \cap A_1 \neq \emptyset$ and $P \cap A_2 \neq \emptyset$, then $P \cap (A_1 \cup A_2)$ is an edge of $P$. Similarly, if $P \cap B_1 \neq \emptyset$ and $P \cap B_2 \neq \emptyset$, then $P \cap (B_1 \cup B_2)$ is an edge of $P$. 
    \label{lem:flat-path-edge}
\end{lem}
\begin{proof}
    Assume that $P \cap A_1 \neq \emptyset$ and $P \cap A_2 \neq \emptyset$. By Lemma \ref{lem:small-intersection}, $|P \cap (A_1 \cup A_2)| \leq 3$. Suppose that $|P \cap A_1| = 2$. Since $P$ is a flat path and $G$ is $C_4$-free, it follows that $|A_2| = 1$. Let $\{a_2\} = A_2$, and note that $a_2 \in P$. Since $a_2$ has two neighbors in $P$, namely $P \cap A_1$, it follows that $a_2$ is an interior vertex of $P$, so by the definition of flat path, $a_2$ has degree two in $G$. But by the definition of 2-join, there is a path with ends in $A_2$ and $B_2$ and interior in $C_2$, so $a_2$ has a neighbor in $B_2 \cup C_2$, a contradiction. This completes the proof. 

\end{proof}

Next, we prove that if a graph $G$ admits a 2-join, then chordal covers of the blocks of decompositions of the 2-join can be combined into a chordal cover of $G$. 
\begin{lem}
\label{lem:helper}
    Let $G$ be a graph with no even hole. Assume that $G$ admits a 2-join $(A_1, C_1, B_1, A_2, C_2, B_2)$, where $Z_i = A_i \cup C_i \cup B_i$ for $i=1,2$. Let $G_1 = B(Z_1)$ and $G_2 = B(Z_2)$, and let $(X_1', X_2')$ and $(X_1'', X_2'')$ be chordal covers of $G_1$ and $G_2$, respectively. Further assume that: 
    \begin{itemize}
        \item $M_2 \subseteq X_1' \cap X_2'$,
        and 
        \item $\{a_1, b_1\} \cap X_1' \subseteq X_1''$ and $\{a_1, b_1\} \cap X_2' \subseteq X_2''$, where $a_1$ and $b_1$ are the ends of $M_1$ in $A_1$ and $B_1$, respectively. 
    \end{itemize}

    Let $X_1 = (X_1' \cap Z_1) \cup (X_1'' \cap Z_2)$ and let $X_2 = (X_2' \cap Z_1) \cup (X_2'' \cap Z_2)$.
    Then, $(X_1, X_2)$ is a chordal cover of $G$.

\end{lem}
\begin{proof}
Since $Z_1 \subseteq G_1$, it holds that $Z_1 \subseteq X_1' \cup X_2'$. Similarly, $Z_2 \subseteq X_1'' \cup X_2''$. Therefore, $Z_1 \cup Z_2 \subseteq X_1 \cup X_2$, and so $X_1 \cup X_2 = V(G)$. We show that $X_i$ is chordal.

    Suppose $H$ is a hole in $X_1$. 
Since $X_1'$ and $X_1''$ are chordal, it follows that $H \not \subseteq Z_1$ and $H \not \subseteq Z_2$; indeed, if say $H  \subseteq Z_1$, then $H \subseteq X_1 \cap Z_1 = X'_1 \cap Z_1 \subseteq X'_1$, contradicting the chordality of $X'_1$.
This implies further that  $H \cap Z_1 \neq \emptyset$ and $H \cap Z_2 \neq \emptyset$. Therefore, $H$ contains an edge with one end in $Z_1$ and one end in $Z_2$.  
By Lemma \ref{lem:small-intersection}, one of the following holds: 
    \begin{itemize}
        \item[(1)]  $H \cap Z_1$ is independent and consists of at most one vertex of $A_1$ and at most one vertex of $B_1$, or
        \item[(2)] $H \cap Z_2$ is independent and consists of at most one vertex of $A_2$ and at most one vertex of $B_2$, or
        \item[(3)] $H \cap Z_1$ is a path with ends in $A_1$ and $B_1$ and (possibly empty) interior in $C_1$ and $H \cap Z_2$ is a path with ends in $A_2$ and $B_2$ and (possibly empty) interior in $C_2$.
    \end{itemize}
    
    First, suppose (1) holds. We claim that if $H \cap A_1 \neq \emptyset$, then $|A_1| = 1$. Suppose for a contradiction that $H \cap A_1 \neq \emptyset$ and $|A_1| > 1$. Note that $|H\cap A_2|>1$ for otherwise $H$ is not a hole. Let $a'$ be the vertex of $H \cap A_1$, and let $a'' \in A_1 \setminus \{a'\}$. Since $\{a', a''\} \cup (N_H(a'))$  is not a $C_4$, it follows that $a'a'' \in E(G)$. But now $(H, a'')$ is a twin wheel of $G$, contradicting that $G$ has no sector wheel. This proves that $|A_1| = 1$, and so $A_1 = \{a_1\}$. Similarly, if $H \cap B_1 \neq \emptyset$, then $B_1 = \{b_1\}$. Therefore, $H \subseteq G_2$. Since $H \subseteq X_1$, it follows that $H \cap Z_1 \subseteq X_1'$, and by the second assumption of the lemma, $H \cap Z_1 \subseteq X_1''$. It follows that $H \subseteq X_1''$, contradicting the chordality of $X_1''$. Therefore, (1) does not hold. 

    Next, suppose (2) holds. Let $H'$ be the hole of $G_1$ formed by replacing the vertex of $H \cap A_2$, if it exists, with $a_2$, and replacing the vertex of $H \cap B_2$, if it exists, with $b_2$. Since $H \subseteq X_1$, it follows that $H \cap Z_1 \subseteq X_1'$, and by the first assumption of the lemma, $H \subseteq X_1'$. Now, $H$ is a hole of $X_1'$, contradicting the chordality of $X_1'$. Therefore, (2) does not hold. 

    Since (1) and (2) do not hold, it follows that (3) holds. Let $Q_1 = H \cap Z_1$ and $Q_2 = H \cap Z_2$, so $Q_1 \subseteq X_1'$. Now, $Q_1 \cup M_2$ is a hole of $G_1$ and, by the first assumption of the lemma, $Q_1 \cup M_2 \subseteq X_1'$, contradicting the chordality of $X_1'$. This completes the proof. 
\end{proof}

By symmetry, the following lemma is also true:

\begin{lem}
\label{lem:helper2}
    Let $G$ be a graph with no even hole. Assume that $G$ admits a 2-join $(A_1, C_1, B_1, A_2, C_2, B_2)$, where $Z_i = A_i \cup C_i \cup B_i$ for $i=1,2$. Let $G_1 = B(Z_1)$ and $G_2 = B(Z_2)$, and let $(X_1', X_2')$ and $(X_1'', X_2'')$ be chordal covers of $G_1$ and $G_2$, respectively. Further assume that: 
    \begin{itemize}
        \item $M_1 \subseteq X_1'' \cap X_2''$ 
        and 
        \item $\{a_2, b_2\} \cap X_1'' \subseteq X_1'$ and $\{a_2, b_2\} \cap X_2'' \subseteq X_2'$, where $a_2$ and $b_2$ are the ends of $M_2$ in $A_2$ and $B_2$, respectively.
    \end{itemize}
    Let $X_1 = (X_1' \cap Z_1) \cup (X_1'' \cap Z_2)$ and let $X_2 = (X_2' \cap Z_1) \cup (X_2'' \cap Z_2)$.
    Then, $(X_1, X_2)$ is a chordal cover of $G$.
\end{lem}

Next, we prove that a partial precover can be extended to a full precover. 

\begin{lem}\label{lem:construct-precoloring}
    Let $G$ be a graph with no even hole, no star cutset, and no sector wheel. Let $P = p_1 \dd \hdots \dd p_k$ be a flat path of $G$. Let $(W_1, W_2)$ be such that $W_1 \cap W_2 = V(P)$, $W_1 \cup W_2 \subseteq N[P]$, and $G[W_1]$ and $G[W_2]$ are chordal. Also assume that $N(p_1) \cap N(p_k) \subseteq W_1 \cup W_2$. Then, there exists $W_1'$, $W_2'$ with $W_1 \subseteq W_1'$, $W_2 \subseteq W_2'$, such that $W_1' \cap W_2' = V(P)$, $W_1' \cup W_2' = N[P]$, and $G[W_1']$ and $G[W_2']$ are chordal. 
\end{lem}
\begin{proof}
     We construct $W_1'$ and $W_2'$ as follows. We begin by adding every vertex of $W_i$ to $W_i'$ for $i = 1, 2$. Then, as long as $N[P] \setminus (W_1' \cup W_2')$ is not empty, we choose $v \in N[P] \setminus (W_1' \cup W_2')$. Note that since $P$ is a flat path and by the assumptions of the lemma, $v\in (N(p_1)\cup N(p_k)) \setminus (N(p_1)\cap N(p_k))$. 
     
     First, suppose that $P$ has length greater than one. Assume $v \in N(p_i) \setminus N(p_{k+1-i})$ for $i \in \{1, k\}$. We claim that $v$ has at most one neighbor in $N(p_{k+1-i})$. Indeed, Suppose $v$  has two neighbors $x_1, x_2 \in N(p_{k+1-i})$. Since $\{p_{k+1-i}, x_1, x_2, v\}$ does not induce a $C_4$, it follows that $x_1$ and $x_2$ are adjacent. If $p_i$ is adjacent to both $x_1$ and $x_2$, then $(P \cup \{x_1\}, x_2)$ is a twin wheel, contradicting that $G$ has no sector wheel. Therefore, we may assume that $p_i$ is non-adjacent to $x_1$.  But now $P \cup \{x_1, v\}$ is a hole and $N(x_2) \cap (P \cup \{x_1, v\})$ is a path of length two, contradicting that $G$ has no sector wheel.

We now follow the process below, which is well defined since $v$ has at most one meighbor in $N(p_{k+1-i})$:

    \begin{itemize}
        \item If $v$ is anticomplete to $N(p_{k+1 - i})$, then add $v$ to $W_1'$. 
        \item If $v$ has a neighbor in $N(p_{k+1-i}) \cap W_1'$, then add $v$ to $W_2'$. 
        \item If $v$ has a neighbor in $N(p_{k+1-i}) \cap W_2'$, then add $v$ to $W_1'$. 
    \end{itemize}

By the above, 
 the sets $W_1'$, $W_2'$ formed in this way are unique: every vertex $v \in N(p_{i}) \setminus N(p_{k+1-i})$ is assigned to exactly one of $W_1', W_2'$ as above. 

Now, suppose that $P$ has length one. Assume $v \in N(p_i) \setminus N(p_{k+i-1})$ for $i \in \{1, k\}$. We claim that $v$ has at most one neighbor in $ N(p_{k+1-i}) \setminus N(p_i)$. Indeed, suppose $v$ has two neighbors $x_1, x_2 \in N(p_{k+1-i}) \setminus N(p_i)$. Since $\{p_{k+1-i}, x_1, x_2, v\}$ does not induce a $C_4$, it follows that $x_1$ and $x_2$ are adjacent. But now $P \cup \{x_1, v\}$ is a hole and $N(x_2) \cap (P \cup \{x_1, v\})$ is a path of length two, contradicting that $G$ has no sector wheel. 

We now follow the process below:
\begin{itemize}
    \item If $v$ is anticomplete to $N(p_{k+1-i}) \setminus N(p_i)$, then add $v$ to $W_1'$. 

    \item If $v$ has a neighbor in $(N(p_{k+1-i}) \setminus N(p_i)) \cap W_1'$, then add $v$ to $W_2'$. 

    \item If $v$ has a neighbor in $(N(p_{k+1-i}) \setminus N(p_i)) \cap W_2'$, then add $v$ to $W_1'$.
\end{itemize}

Again, every vertex $v \in N(p_i) \setminus N(p_{k+1-i})$ is assigned to exactly one of $W_1'$, $W_2'$ by the argument above.

Next, we prove that $W_1'$ and $W_2'$ satisfy the conditions of the lemma. By the construction of $W_1'$ and $W_2'$, we have that $W_1' \cap W_2' = V(P)$ and that $W_1' \cup W_2' = N[P]$. It remains to show that $G[W_1']$ and $G[W_2']$ are chordal. Suppose that $G[W_1']$ contains a hole $H$. Suppose first that $P \subseteq H$; so $H \setminus P$ is either an edge with one end in $N(p_1)$ and one end in $N(p_k)$ or a vertex in $N(p_1) \cap N(p_k)$. Since, by the construction of $W_1'$ and $W_2'$, no edge with one end in $N(p_1) \setminus N(p_k)$ and one end in $N(p_k) \setminus N(p_1)$ has both ends in $W_1'$, it follows that $H \setminus P$ is a vertex in $N(p_1) \cap N(p_k)$. But now $H$ is a hole of $W_1$, a contradiction. Therefore, $P \not \subseteq H$, and thus $H \subseteq N[p_1]$ or $H \subseteq N[p_k]$. Since $p_i$ is complete to $N(p_i)$, it follows that $H \subseteq N(p_i)$ for $i = 1, k$. But now $(H, p_i)$ is a universal wheel, a contradiction. This proves that $G[W_1']$ is chordal. The proof that $G[W_2']$ is chordal follows similarly. 
\end{proof}

Next, we prove: 

\begin{lem}\label{lem:2-join}
Let $G$ be a graph with no even hole, no sector wheel, and no star cutset. Suppose that $G$ is non-FPE and that every proper induced subgraph of $G$ with no star cutset is FPE. Then, $G$ does not admit a 2-join. 
\end{lem}
\begin{proof}
    Suppose for a contradiction that $(A_1, C_1, B_1, A_2, C_2, B_2)$ is a 2-join of $G$. Let $P = p_1 \dd \hdots \dd p_k$ be a witness path for $G$ with witness sets $W_1$, $W_2$. Assume up to symmetry that $p_1 \in Z_1=A_1 \cup C_1 \cup B_1$. Let $G_1 = B(Z_1)$ and let $G_2 = B(Z_2)$. By Theorem \ref{thm:block-no-sc}, $G_1$ and $G_2$ have no star cutset. Since every proper induced subgraph of $G$ with no star cutset is FPE, and $G_1$ and $G_2$ are proper induced subgraphs of $G$ with no star cutsets, it follows that $G_1$ and $G_2$ are FPE. Our strategy to complete the proof is to find appropriate chordal covers of $G_1$ and $G_2$, and use Lemmas \ref{lem:helper} and \ref{lem:helper2} to obtain a chordal cover of $G$ that extends $(P, W_1, W_2)$, reaching a contradiction. 

    First we show: 

    \sta{\label{pk-in-Z1} $p_k \not \in Z_2$. }

    Suppose that $p_k \in Z_2= A_2 \cup C_2 \cup B_2$. By Lemma \ref{lem:flat-path-edge}, $P$ contains exactly one edge with one end in $Z_1$ and one end in $Z_2$. Assume up to symmetry between $A_1$ and $B_1$ that $1 \leq i \leq k$ is such that $\{p_1, \hdots, p_i\} \subseteq Z_1$, $\{p_{i+1}, \hdots, p_k\} \subseteq Z_2$, $p_i \in A_1$, and $p_{i+1} \in A_2$. Since $B_1$ is complete to $B_2$, not both $P \cap B_1 \neq \emptyset$ and $P \cap B_2 \neq \emptyset$, and since $p_1\in Z_1$ and $p_k\in Z_2$, we may assume up to symmetry between $B_1$ and $B_2$ that $P \cap B_1 = \emptyset$.

    Let $P_1 = (P \cap Z_1) \cup M_2$. Let $W_1'' = N(p_1) \cap W_1 \cap G_1$ and $W_2'' = N(p_1) \cap W_2 \cap G_1$. By Lemma \ref{lem:construct-precoloring}, there exist sets $W_1'$, $W_2'$ such that $W_1' \cap W_2' = V(P_1)$, $W_1' \cup W_2' = N[P_1]$, and $G_1[W_1']$ and $G_1[W_2']$ are chordal. Let $(X_1', X_2')$ be a chordal cover of $G_1$ that extends $(P_1, W_1', W_2')$. Next, we find a chordal cover of $G_2$.  First, assume that $P \cap (B_1 \cup B_2) = \emptyset$. Let $P_2 = (P \cap Z_2) \cup M_1$.  Let $Y_1'' = N(p_k) \cap W_1 \cap G_2$ and $Y_2'' = N(p_k) \cap W_2 \cap G_2$. By Lemma \ref{lem:construct-precoloring}, there exist sets $Y_1', Y_2'$ such that $Y_1' \cap Y_2' = V(P_2)$, $Y_1' \cup Y_2' = N[P_2]$, and $G_2[Y_1']$ and $G_2[Y_2']$ are chordal. Let $(X_1'', X_2'')$ be a chordal cover of $G_2$ that extends $(P_2, Y_1', Y_2')$. Let $X_1 = (Z_1 \cap X_1') \cup (Z_2 \cap X_1'')$ and let $X_2 = (Z_1 \cap X_2') \cup (Z_2 \cap X_2'')$. Since $M_2 \subseteq P_1$ and $M_1 \subseteq P_2$, it follows that the conditions of Lemma \ref{lem:helper} are satisfied. By Lemma \ref{lem:helper}, $(X_1, X_2)$ is a chordal cover of $G$. By the construction of $X_1$ and $X_2$, it follows that $X_1 \cap X_2 = V(P)$, $W_1 \subseteq X_1$, and $W_2 \subseteq X_2$. Now, $(X_1, X_2)$ is a chordal cover of $G$ that extends $(P, W_1, W_2)$, contradicting that $G$ is non-FPE. 

    Therefore, $P \cap B_2 \neq \emptyset$. Let $a_1$ and $b_1$ be the ends of $M_1$. Let $P_2' = (P \cap Z_2)$. Let $U_1'' = (N(p_k) \cap W_1 \cap G_2) \cup (\{a_1, b_1\} \cap X_1')$ and let $U_2'' = (N(p_k) \cap W_2 \cap G_2) \cup (\{a_1, b_1\} \cap X_2')$. By Lemma \ref{lem:construct-precoloring}, there exist sets $U_1', U_2'$ such that $U_1' \cap U_2' = V(P_2')$, $U_1' \cup U_2' = N[P_2']$, and $G_2[U_1']$ and $G_2[U_2']$ are chordal. Let $(X_1'', X_2'')$ be a chordal cover of $G_2$ that extends $(P_2', U_1', U_2')$. Let $X_1 = (X_1' \cap Z_1) \cup (X_1'' \cap Z_2)$ and let $X_2 = (X_2' \cap Z_1) \cup (X_2'' \cap Z_2)$. Since $M_2 \subseteq P_1$, and by construction of $U_1''$ and $U_2''$, it follows that the conditions of Lemma \ref{lem:helper} are satisfied. Now, by Lemma \ref{lem:helper}, $(X_1, X_2)$ is a chordal cover of $G$. By the construction of $X_1$ and $X_2$, it follows that $X_1 \cap X_2 = V(P), W_1 \subseteq X_1$, and $W_2 \subseteq X_2$. Now, $(X_1, X_2)$ is a chordal cover of $G$ that extends $(P, W_1, W_2)$, contradicting that $G$ is non-FPE. This proves \eqref{pk-in-Z1}. \\

Next we show: 

    \sta{\label{P-in-Z1} $P \subseteq Z_1$.}

    By \eqref{pk-in-Z1}, $\{p_1, p_k\} \subseteq Z_1$. Suppose $P \not \subseteq Z_1$. By Lemma \ref{lem:flat-path-edge}, it follows that $P \cap Z_2$ is a path with ends in $A_2$ and $B_2$ and interior in $C_2$. Let $P_1 = (P \cap Z_1) \cup M_2$ and let $P_2 = M_1$. Let $W_1'' = N(P_1) \cap W_1 \cap G_1$ and $W_2'' = N(P_1) \cap W_2 \cap G_2$. By Lemma \ref{lem:construct-precoloring}, there exist $W_1'$, $W_2'$ such that $W_1' \cap W_2' = V(P_1)$, $W_1' \cup W_2' = N[P_1]$, and $G_1[W_1']$ and $G_1[W_2']$ are chordal. Let $(X_1', X_2')$ be a chordal cover of $G_1$ that extends $(P_1, W_1', W_2')$. Next, let $U_1'' = N(P_2) \cap W_1 \cap G_2$ and $U_2'' = N(P_2) \cap W_2 \cap G_2$. By Lemma \ref{lem:construct-precoloring}, there exists $U_1'$, $U_2'$ such that $U_1' \cap U_2' = V(P_2)$, $U_1' \cup U_2' = N[P_2]$, and $G_2[U_1']$ and $G_2[U_2']$ are chordal. Let $(X_1'', X_2'')$ be a chordal cover of $G_2$ that extends $(P_2, U_1', U_2')$. 

    Now, let $X_1 = (Z_1 \cap X_1') \cup (Z_2 \cap X_1'')$ and $X_2 = (Z_1 \cap X_2') \cup (Z_2 \cap X_2'')$. Since $M_2 \subseteq P_1$ and $M_1 \subseteq P_2$, it follows that the conditions of Lemma \ref{lem:helper} are satisfied. By Lemma \ref{lem:helper}, $(X_1, X_2)$ is a chordal cover of $G$. By the construction of $X_1$ and $X_2$, it follows that $X_1 \cap X_2 = V(P)$, $W_1 \subseteq X_1$, and $W_2 \subseteq X_2$. Therefore, $(X_1, X_2)$ is a chordal cover of $G$ that extends $(P, W_1, W_2)$, contradicting that $G$ is non-FPE. This proves \eqref{P-in-Z1}. \\

Next we show: 

\sta{\label{P-in-C1} $P \subseteq C_1$.}
 By \eqref{P-in-Z1}, $P \subseteq Z_1=A_1 \cup C_1 \cup B_1$. Let $P_2 = M_1$, let $W_1'' = N(P_2) \cap W_1 \cap G_2$, and let $W_2'' = N(P_2) \cap W_2 \cap G_2$. By Lemma \ref{lem:construct-precoloring}, there exist $W_1', W_2'$ such that $W_1' \cap W_2' = V(P_2)$, $W_1' \cup W_2' = N[P_2]$, and $G_2[W_1']$ and $G_2[W_2']$ are chordal. Let $(X_1'', X_2'')$ be a chordal cover of $G_2$  that extends $(P_2, W_1', W_2')$.

 Since $P \not \subseteq C_1$, either $P \cap A_1 \neq \emptyset$ or $P \cap B_1 \neq \emptyset$; by symmetry, assume that $P \cap A_1 \neq \emptyset$. 
If $P \cap B_1 = \emptyset$, let $P_1 = P \cup M_2$. If $P \cap B_1 \neq \emptyset$, let $P_1 = P$.
 Let $U_1'' = N(P_1) \cap W_1 \cap G_1$ and $U_2'' = N(P_1) \cap W_2 \cap G_1$. Note that in both cases, the second condition of Lemma \ref{lem:helper2} holds (in the first case, because $M_2 \subseteq P_1$ and so $\{a_2, b_2\} \subseteq X_1'' \cap X_2''$, and in the second case, because $\{a_2, b_2\} \subseteq N(P_1) \subseteq W_1 \cup W_2$). By Lemma \ref{lem:construct-precoloring}, there exist $U_1', U_2'$ such that $U_1' \cap U_2' = V(P_1)$, $U_1' \cup U_2' = N[P_1]$, and $G_1[U_1']$ and $G_1[U_2']$ are chordal. Let $(X_1', X_2')$ be a chordal cover of $G$ that extends $(P_1, U_1', U_2')$. 

Let $X_1 = (Z_1 \cap X_1') \cup (Z_2 \cap X_1'')$ and $X_2 = (Z_1 \cap X_2') \cup (Z_2 \cap X_2'')$. By Lemma \ref{lem:helper2}, $(X_1, X_2)$ is a chordal cover of $G$. By the construction of $X_1$ and $X_2$, it follows that $X_1 \cap X_2 = V(P)$, $W_1 \subseteq X_1$, and $W_2 \subseteq X_2$. Now, $(X_1, X_2)$ is a chordal cover of $G$ that extends $(P, W_1, W_2)$, contradicting that $G$ is non-FPE. This proves \eqref{P-in-C1}. \\

By \eqref{P-in-C1}, $P \subseteq C_1$. Therefore, $W_1, W_2 \subseteq Z_1$. Let $(X_1', X_2')$ be a chordal cover of $G_1$ that extends $(P, W_1, W_2)$. Let $P_2 = M_1$, let $W_1'' = ((\{a_2, b_2\}) \cap X_1') \cup M_1$, and let $W_2'' = ((\{a_2, b_2\}) \cap X_2') \cup M_1$. Note that since $X_1'$ and $X_2'$ are chordal, it follows that $G_2[W_1''
$ and $G_2[W_2'']$ are chordal. By Lemma \ref{lem:construct-precoloring}, there exist sets $W_1', W_2'$ such that $W_1' \cap W_2' = V(P_2)$, $W_1' \cup W_2' = N[P_2]$, and $G_2[W_1']$ and $G_2[W_2']$ are chordal. Let $(X_1'', X_2'')$ be a chordal cover of $G_2$ that extends $(P_2, W_1', W_2')$. Let $X_1 = (Z_1 \cap X_1') \cup (Z_2 \cap X_1'')$ and $X_2 = (Z_1 \cap X_2') \cup (Z_2 \cap X_2'')$. The conditions of Lemma \ref{lem:helper2} are satisfied by the construction of $P_2$, $W_1''$, and $W_2''$, so by Lemma \ref{lem:helper2}, $(X_1, X_2)$ is a chordal cover of $G$ that extends $(P, W_1, W_2)$, contradicting that $G$ is non-FPE. This completes the proof. 
\end{proof}

Finally, we prove the main result of this section:

\begin{thm}
\label{thm:MNFPE}
    Let $G$ be an even-hole-free graph with no sector wheel and no star cutset. If every proper induced subgraph of $G$ with no star cutset is FPE, then $G$ is FPE.  
\end{thm}
\begin{proof}
 We apply Theorem \ref{thm:decomposition} to $G$. If $G$ is a clique, then $G$ is chordal, so every precover of $G$ can be arbitrarily extended to a chordal cover of $G$ and thus $G$ is FPE. If $G$ is a hole, then every precover of $G$ can be extended to a chordal cover $(X_1, X_2)$ of $G$ by ensuring that both $X_1 \setminus X_2$ and $X_2 \setminus X_1$ are non-empty, so $G$ is FPE. Suppose $G$ is a pyramid with base $b_1b_2b_3$, apex $a$, and paths $P_1, P_2, P_3$. Let $P$ be a witness path of $G$. Up to symmetry, we may assume that $P$ is contained in $P_1$. Then, every precover of $G$ with witness path $P$ can be extended to a chordal cover $(X_1, X_2)$ by ensuring that both $P_2$ and $P_3$ meet both $X_2 \setminus X_1$ and $X_1 \setminus X_2$. It follows that $G$ is FPE. 

 Therefore, we may assume that either $G$ is an extended nontrivial basic graph or $G$ admits a 2-join. By Lemma \ref{lem:2-join}, $G$ does not admit a 2-join. If $G$ is an extended nontrivial basic graph, then $G$ is FPE by Lemma \ref{lem:basic}.  This completes the proof. 
\end{proof}

\section{Graphs with a clique cutset} 

In this section, we prove that even-hole-free graphs with no sector wheels that have a clique cutset are minimal non-weakly FPE. 
\label{sec:clique}
\begin{lem}\label{staypath}
Let $G$ be a graph. Suppose $G$ has a  clique cutset $Q$, and let $C$ be a component of $G-Q$. Let  $G' = G[C \cup Q]$ and $G''=G-C$. If $P$ is a flat path in $G$, then $P\cap G'$ and $P\cap G''$ are paths or empty.
\end{lem}
\begin{proof}
If not, then there exists two non-adjacent vertices in $Q$, a contradiction. 
\end{proof}

\begin{lem}
\label{lem:no-clique-cutset}

Let $G$ be a minimal non-weakly FPE graph. Then, $G$ does not have a clique cutset. 
\end{lem}
\begin{proof}
Let $Q$ be a clique cutset. Let $P$ be  a flat path witnessing $G$, and let $W_1,W_2$ be the corresponding witness sets. First, we prove:

\sta{\label{clique-helper} Let $Q$ be a clique cutset, let $C$ be a component of $G - C$, let $G' = G[C \cup Q]$, and let $G'' = G - C$. Suppose $X \subseteq G'$ is chordal and $Y \subseteq G''$ is chordal. Then, $G[X \cup Y]$ is chordal.}

Suppose there is an induced cycle $T$ in $G[X \cup Y]$. Then, $|T \cap Q| \geq 2$, otherwise $T \subseteq X$ or $T \subseteq Y$. Since $Q$ is a clique, it follows that $|T \cap Q| = 2$, and since $T$ is a cycle, it follows that $T \setminus Q \subseteq X$ or $T \setminus Q \subseteq Y$. But $T \cap Q \subseteq X \cap Y$, so $T \subseteq X$ or $T \subseteq Y$, contradicting that $X$ and $Y$ are chordal. This proves \eqref{clique-helper}. \\

First, suppose there exists a component $C$ of $G - Q$ such that $(W_1 \cup W_2) \cap C = \emptyset$. Let  $G' = G[C \cup Q]$ and $G''=G-C$. Then, $W_1\cup W_2\subseteq V(G'')$, and $P\cap G''$ is a flat path.  By the induction hypothesis, the chordal cover $W_1\cup W_2$  can be extended to a chordal cover $X_1\cup X_2$ of $G''$. Choose a vertex $v\in Q$ and think of $v$ as a flat path $P'$ and of $Q-v$ as a subset of $N[P']$ in $G'$. 
Let $W'_i=(X_i\cap Q)\cup \{v\}$, for $i=1,2$. 
By the induction hypothesis $W'_1\cup W'_2$ is extendable to a chordal cover $Y_1\cup Y_2$  of  $G'$.  Remove $v$ from $Y_i$ if it is not in $X_i$, for $i=1,2$. By \eqref{clique-helper}, $G[X_i\cup Y_i]$ is chordal for $i=1,2$. Now, $(P, X_1 \cup Y_1, X_2 \cup Y_2)$ is a chordal cover of $G$ that extends $(P, W_1, W_2)$, a contradiction. 

Therefore, we may assume that $W_1\cup W_2$
intersects every component of $G-Q$. Let $C$ be a component of $G - Q$, let $G' = G[C \cup Q]$, and let $G'' = G - C$. By Lemma \ref{staypath}, 
 $P'= P\cap G'$ and  $P''= P\cap G''$ are flat paths. For $i=1,2$, let $W'_i = W_i \cap V(G')$, $W''_i = W_i \cap V(G'')$. By the induction hypothesis,  $W'_1\cup W'_2$  and $W''_1\cup W''_2$  can be extended to chordal covers $X_1\cup X_2$ and $Y_1\cup Y_2$ of $G'$ and $G''$, respectively.
The sets $G[X_i\cup Y_i]$ are chordal for $i = 1, 2$ by \eqref{clique-helper}. It follows that $(P, X_1 \cup Y_1, X_2 \cup Y_2)$ is a chordal cover of $G$ that extends $(P, W_1, W_2)$, a contradiction. This completes the proof. 
 \end{proof}

\section{Graphs with a proper star cutset}\label{sec:star-cutset}

In this section we prove that even-hole-free graphs with no sector wheels that have proper star cutsets are minimal non-weakly FPE. 
We begin with a few useful lemmas.

\begin{lem}
\label{lem:v-not-full}
Let $G$ be minimal non-weakly FPE and let $v \in V(G)$. Then, $v$ is not complete to $G \setminus \{v\}$. 
\end{lem}
\begin{proof}
Let $P$ be the witness path and $(W_1, W_2)$ be the witness sets for $G$. Suppose for the sake of contradiction that $v$ is complete to $G' = G \setminus \{v\}$. Since $G$ is minimal non-weakly FPE, $G = N[v]$, and $W_1 \cup W_2 = N[P]$, it follows that $v \not \in P$. Since $G$ is minimal non-weakly FPE and $G'$ is a proper induced subgraph of $G$, it follows that $G'$ is weakly FPE. Since $P \subseteq V(G')$, there exists a chordal cover $(X_1, X_2)$ of $G'$ that extends $(P, W_1 \setminus \{v\}, W_2 \setminus \{v\})$. 
Now, since $v$ is complete to $G'$ we have $v \in N[P]$, and we may assume up to symmetry that $v \in W_1$. Since $v$ is complete to $X_1$ and $X_1$ is chordal, it follows that $X_1 \cup \{v\}$ is chordal, so $(X_1 \cup \{v\}, X_2)$ is a chordal cover of $G$ that extends $(P, W_1, W_2)$, a contradiction. 
\end{proof}

\begin{lem}
\label{lem:one-comp-anticomplete}
Let $G$ be a graph and let $X \subseteq V(G)$ be a subset of its vertex set such that there exists a vertex $u \in V(G)$ anticomplete to $X$. Suppose there exists a cutset $Y$ of $G$ such that $X \subseteq Y \subseteq N[X]$. Then, there exists a cutset $Y'$ of $G$ such that $X \subseteq Y' \subseteq N[X]$ and at least one component of $G \setminus Y'$ is anticomplete to $X$. 
\end{lem}
\begin{proof}
Let $C_1, \hdots, C_m$ be the components of $G \setminus Y$. Since $u$ is anticomplete to $X$, it follows that $u \not \in Y$, and so we may assume up to symmetry that $u \in C_1$. Let $Y' = Y \cup \left(\bigcup_{v \in X} (N(v) \cap C_1)\right)$.  Let $C_u$ be the component of $G \setminus Y'$ containing $u$. Now, $C_u$ is anticomplete to $X$. This completes the proof. 
\end{proof}

A {\em twin wheel} consists of a hole $H$ and a vertex $v$ such that $H \cap N(v)$ is a three-vertex path. A {\em short pyramid} consists of a hole $H$ and a vertex $v$ such that $H \cap N(v)$ is an edge plus an isolated vertex. For a path $P = p_1 \dd \hdots \dd p_k$, let $P^*$ denote the {\em interior} of $P$; that is, $P^* = P \setminus \{p_1, p_k\}$. A wheel is {\em proper} if it is not a twin wheel or a short pyramid. A wheel $(H, v)$ is  {\em universal} if $v$ is complete to $H$. 
A {\em sector} of a wheel $(H, v)$ is a path $P \subseteq H$ such that $v$ is complete to the ends of $P$ and anticomplete to the interior of $P$. A sector is {\em long} if it has length greater than one. A wheel $(H, v)$ is called an {\em even wheel} if $|N(v) \cap H|$ is even. If $H$ is a graph, then we say that $G$ {\em contains $H$} if $G$ has an induced subgraph isomorphic to $H$. 

The following is well-known; we include a proof for completeness. 

\begin{lem}
\label{lem:no-even-wheel}
Let $G$ be a graph with no even hole. Then, $G$ does not contain an even wheel. 
\end{lem}
\begin{proof}
Suppose $G$ contains an even wheel $(H, v)$, and  suppose $S$ is a long sector of $(H, v)$.  Then, $S \cup \{v\}$ is a hole of $G$ whose length is the same parity as the length of $S$. It follows that every long sector is of odd length. Since sectors that are not long are of length one, it follows that every sector of $(H, v)$ is of odd length. Since $(H, v)$ is an even wheel, $(H, v)$ has an even number of sectors. But now $H$ is even, a contradiction. 
\end{proof}

The following lemma describes star cutsets that come from proper wheel centers. 

\begin{lem}[\cite{Addario-Berry2008BisimplicialGraphs, daSilva2013Decomposition2-joins}]
\label{lem:proper-wheel} Let $G$ be a graph with no even hole that contains a proper wheel $(H, x)$
that is not a universal wheel. Let $x_1$ and $x_2$ be the endpoints of a long sector $Q$ of $(H, x)$. Let $W$
be the set of all vertices $h \in H \cap N(x)$ such that the subpath of $H \setminus \{x_1\}$ from $x_2$ to $h$ contains
an even number of neighbors of $x$, and let $Z = H \setminus (Q \cup N(x))$. Let $N' = N(x) \setminus W$. Then,
$N' \cup \{x\}$ is a cutset of $G$ that separates $Q^*$
from $W \cup Z$.
\end{lem}
We will also use the following corollary of Lemma \ref{lem:proper-wheel}: 

\begin{lem}
\label{lem:short-pyramid}
Let $G$ be a graph with no even hole and no twin wheel, and let $(H, x)$ be a wheel of $G$. Suppose $x$ is not the center of a star cutset in $G$. Then, $(H, x)$ is a short pyramid. 
\end{lem}
\begin{proof}
Suppose $(H, x)$ is not a short pyramid. Since $G$ has no twin wheel, it follows that $(H, x)$ is a proper wheel. By Lemma \ref{lem:proper-wheel}, it follows that $x$ is the center of a star cutset in $G$, a contradiction. 
\end{proof}

Next, we prove a helpful lemma about cutsets contained in the neighborhood of witness paths. 

\begin{lem}
\label{lem:star-cutset-helper} Let $G$ be minimal non-weakly FPE, let $P$ be a witness path for $G$ with witness sets $W_1$ and $W_2$, and let $X$ be a cutset of $G$ such that $X \cap P$ is connected and $X \subseteq N[X \cap P]$. Then, no component of $G \setminus X$ is anticomplete to $X \cap P$.
\end{lem}
\begin{proof}
Let $C_1, \hdots, C_m$ be the components of $G \setminus X$, and suppose for the sake of contradiction that $C_1$ is anticomplete to $X \cap P$. Let $G' = X \cup C_1$ and $G'' = G \setminus C_1$. Note that $X \cap P$ is a flat path in $G'$, $X \cap P \subseteq (W_1 \cap G') \cap (W_2 \cap G')$, and $(W_1 \cap G') \cup (W_2 \cap G') = N_{G'}[X \cap P]$. Similarly,  $X \cap P$ is a flat path in $G''$, $X \cap P \subseteq (W_1 \cap G'') \cap (W_2 \cap G'')$, and $(W_1 \cap G'') \cup (W_2 \cap G'') = N_{G''}[X \cap P]$.  Since $G$ is minimal non-weakly FPE and $G'$ and $G''$ are proper induced subgraph of $G$, it follows that there exists a chordal cover $(X_1', X_2')$ of $G'$ that extends $(X \cap P, W_1 \cap G', W_2 \cap G')$ and a chordal cover $(X_1'', X_2'')$ of $G''$ that extends $(X \cap P, W_1 \cap G'', W_2 \cap G'')$. Let $X_1 = X_1' \cup X_1''$ and let $X_2 = X_2' \cup X_2''$. We claim that $X_1$ and $X_2$ are chordal.

Suppose that there is a hole $H\subseteq X_1$. Since $X_1''$ is chordal, it follows that $H \not \subseteq X_1''$, and so $H \cap C_1 \neq \emptyset$. Let $H' = H \cap N[C_1]$. Since $X_1'$ is chordal, it follows that $H \not \subseteq X_1'$. Since $N[C_1] \subseteq X_1'$, $H \not \subseteq X_1'$, and $H \cap C_1 \neq \emptyset$, it follows that $H'$ contains a path $Q = q_1 \dd \hdots \dd q_k$ with interior $Q^*$ in $C_1$ and ends $q_1,q_k \in N(C_1) \subseteq X\subseteq N[X \cap P]$. 
Now, since $P$ is anticomplete to $Q^*$, $Q \cup P$ contains a hole $\tilde{H}$  and $\tilde{H} \subseteq X_1'$,  contradicting that $(X_1', X_2')$ is a chordal cover of $G'$. Therefore, $X_1$ is chordal, and by symmetry, $X_2$ is chordal. Note that $W_1 \subseteq X_1$ and $W_2 \subseteq X_2$. Thus $(X_1, X_2)$ is a chordal cover of $G$ that extends $(P, W_1, W_2)$, a contradiction. 
\end{proof}

A set $X \subseteq V(G)$ is a {\em full star cutset} if $X$ is a star cutset and $X = N[v]$ for some $v \in V(G)$. If $X = N[v]$ is a full star cutset, the vertex $v$ is called the {\em center} of the full star cutset. A set $X \subseteq V(G)$ is a {\em double star cutset} if there exist $u, 
v \in V(G)$ such that $uv \in E(G)$ and $\{u, v\} \subseteq X \subseteq N[\{u, v\}]$. 

The next lemma is the main result of this section.
\begin{lem}
\label{lem:no-full-star-cutset}
Let $G$ be a graph with no even hole and no twin wheel. Suppose $G$ is minimal non-weakly FPE, and let $P = v_0w_0$ be a witness path of length one with witness sets $W_1$ and $W_2$ such that $W_1 \cup W_2 = N[P]$. Then $G$ does not admit a full star cutset. 
\end{lem}
\begin{proof}  
We start by proving a few claims.

\sta{\label{not-star-cutsets} $v_0$ and $w_0$ are not centers of star cutsets of $G$.}

Suppose $v_0$ is the center of a star cutset $Y \subseteq N[v_0]$
and let $C_1, \hdots, C_m$ be the components of $G \setminus Y$. By Lemma \ref{lem:v-not-full}, $v_0$ is not complete to $G \setminus \{v_0\}$. Thus, applying Lemma \ref{lem:one-comp-anticomplete} with $X=\{v_0\}$, we may assume that $C_1$ is anticomplete to $v_0$. However, by Lemma \ref{lem:star-cutset-helper}, no component of $G \setminus Y$ is anticomplete to $\{v_0\}$, a contradiction. This proves \eqref{not-star-cutsets}.

\sta{\label{not-double-star-cutset}$v_0w_0$ is not the center of a double star cutset of $G$.}

Suppose there exists a cutset $X \subseteq N[\{v_0, w_0\}]$ of $G$ with $\{v_0, w_0\} \subseteq X$ and let $C_1, \hdots, C_m$ be the components of $G \setminus X$. If $G \subseteq N[\{v_0, w_0\}]$, then $(W_1, W_2)$ is a chordal cover of $G$, a contradiction, so $G \not \subseteq N[\{v_0, w_0\}]$.
By Lemma \ref{lem:one-comp-anticomplete}, we may assume that $C_1$ is anticomplete to $\{v_0, w_0\}$. However, by Lemma \ref{lem:star-cutset-helper}, no component of $G \setminus X$ is anticomplete to $\{v_0, w_0\}$, a contradiction. This proves \eqref{not-double-star-cutset}. \\

Suppose for the sake of contradiction that $v \in V(G)$ is the center of a full star cutset $N[v]$ in $G$. By \eqref{not-star-cutsets}, $v \not \in \{v_0, w_0\}$. Let $C_1, \hdots, C_m$ be the connected components of $G \setminus N[v]$. 

\sta{\label{ha} $v_0$ has a neighbor in $C_i$ for $1 \leq i \leq m$. Similarly, $w_0$ has a neighbor in $C_i$ for $1 \leq i \leq m$. } 

First, suppose that $\{v_0, w_0\} \cap N(v) = \emptyset$. We may assume that $\{v_0, w_0\} \subseteq C_1$. Let $G' = C_1 \cup N[v]$ and note that $N[\{v_0, w_0\}] \subseteq G'$. Since $G'$ is a proper induced subgraph of $G$ and $G$ is minimal non-weakly FPE, it follows that $G'$ is weakly FPE. Note that $W_1 \cup W_2\subseteq G'$. Let $(X_1', X_2')$ be a chordal cover of $G'$ that extends $(P, W_1, W_2)$. 

Next, let $G'' = G \setminus C_1$. Let $W_1'' = (X_1' \cap N[v]) \cup \{v\}$ and $W_2'' = (X_2' \cap N[v]) \cup \{v\}$. We think of $v$ as a flat path in $G''$, and note that $W_1'' \cup W_2'' = N[v]$. Since $G''$ is a proper induced subgraph of $G$ and $G$ is minimal non-weakly FPE, it follows that $G''$ is weakly FPE. Let $(X_1'', X_2'')$ be a chordal cover of $G''$ that extends $(v, W_1'', W_2'')$. Let $X_1 = X_1' \cup (X_1'' \setminus \{v\})$ and let $X_2 = X_2' \cup (X_2'' \setminus \{v\})$. We claim that $(X_1, X_2)$ is a chordal cover of $G$ that extends $(P, W_1, W_2)$. Suppose for contradiction that there is a hole $H \subseteq X_1$. Since $X_1'$ and $X_1''$ are chordal, it follows that $H \cap (X_1'' \setminus X_1') \neq \emptyset$ and $H \cap (X_1' \setminus X_1'') \neq \emptyset$. So there exists a path $Q \subseteq X_1'' \setminus X_1'$ in  $H$ with interior in $C_i$ and ends in $N(v)$ for some $1 < i \leq m$.  But now $Q \cup \{v\}$ is a hole in $X_1''$, a contradiction. By the same argument, there is no hole $H \subseteq X_2$. This is a contradiction to the fact that $G$ is minimal non-weakly FPE. Therefore, $\{v_0, w_0\} \cap N(v) \neq \emptyset$, and  we may assume that $w_0 \in N(v)$.  

Suppose $v_0$ is anticomplete to  $C_i$ for some $1 \leq i \leq m$. Let $G' = G \setminus C_i$. Now, $P \subseteq G'$ and $G'$ is a proper induced subgraph of $G$. Since $G$ is minimal non-weakly FPE, it follows that $G'$ is weakly FPE, so there exists a chordal cover $(X_1', X_2')$ of $G'$ that extends $(P, W_1 \cap G', W_2 \cap G')$. Next, let $G'' = C_i \cup N[v]$ and let $P'' = vw_0$.  Let $W_1'' = (N[v] \cap X_1') \cup (W_1 \cap N[w_0] \cap G'') \cup \{v, w_0\}$ and let $W_2'' = (N[v] \cap X_2') \cup (W_2 \cap N[w_0] \cap G'') \cup \{v, w_0\}$. Note that by definition, $W_1'' \cup W_2'' = N[\{v, w_0\}] \cap G''$ and $\{v, w_0\} \subseteq W_1'' \cap W_2''$. Since $G''$ is a proper induced subgraph of $G$, it follows that $G''$ is weakly FPE.  Let $(X_1'', X_2'')$ be a chordal cover of $G''$ that extends $(P'', W_1'', W_2'')$. 

Let $X_1 = X_1' \cup (X_1'' \setminus \{v\})$ and let $X_2 = X_2' \cup (X_2'' \setminus \{v\})$. We claim that $(X_1, X_2)$ is a chordal cover of $G$ that extends $(P, W_1, W_2)$. Suppose for a contradiction that there is a hole $H \subseteq X_1$. Since $X_1'$ is chordal, it follows that $H \cap (X_1'' \setminus X_1') \neq \emptyset$, so $H$ contains a path $Q$ with ends in $N[v]$ and interior in $G \setminus G'$. But now $Q \cup \{v\}$ is a hole and $Q \cup \{v\} \subseteq X_1''$, a contradiction. It follows that $X_1$ is chordal, and by symmetry, $X_2$ is chordal. Now, $(X_1, X_2)$ is a chordal cover of $G$ that extends $(P, W_1, W_2)$, a contradiction. Therefore, $v_0$ has a neighbor in $C_i$ for $1 \leq i \leq m$, and so in particular, $v_0 \in N(v)$. Now the same proof using $P' = vv_0$ shows that $w_0$ has a neighbor in $C_i$ for $1 \leq i \leq m$. This proves \eqref{ha}. \\

By \eqref{ha}, $\{v_0, w_0\} \subseteq N(v)$ and by \eqref{not-double-star-cutset}, $\{v_0, w_0\}$ is not the center of a double star cutset of $G$, so for all $1 \leq i \leq j \leq m$, there exists a path $Q = q_1 \dd \hdots \dd q_k$  from $C_i$ to $C_j$ that is anticomplete to $\{v_0, w_0\}$ such that $q_1 \in C_i$, $q_k \in C_j$, $Q^* \subseteq N(v)$. 
Let $Q = q_1 \dd \hdots \dd q_k$ be the shortest such path. We may assume up to symmetry that $i = 1$ and $j = 2$. Let $R \subseteq C_1$ be the shortest path with one end $q_1$ such that $R$ contains neighbors of both $v_0$ and $w_0$. Similarly, let $S \subseteq C_2$ be the shortest path with one end $q_k$ such that $S$ contains neighbors of both $v_0$ and $w_0$. (Note that both $R$ and $S$ exist by \eqref{ha}). Let $R = q_1 \dd r_1 \dd \hdots \dd r_\ell$ and let $S = q_k \dd s_1 \dd \hdots \dd s_t$. Since $R$ is the shortest path containing neighbors of both $v_0$ and $w_0$, it follows that $R \setminus \{r_\ell\}$ contains neighbors of at most one of $v_0$ and $w_0$. Similarly, $S \setminus \{s_t\}$ contains neighbors of at most one of $v_0$ and $w_0$. We may assume that $r_\ell$ is the unique neighbor of $v_0$ in $R$. 

\sta{\label{w-in-R} $w_0$ has exactly one neighbor $r_w$ in $R$ and $r_w \neq r_\ell$.} 

Let $H_1$ be the hole given by $H_1 = v_0 \dd v \dd q_2 \dd q_1 \dd R \dd r_\ell \dd v_0$. Since $R$ contains neighbors of $w_0$, it follows that $w_0$ has at least three neighbors in $H_1$: $v_0$, $v$, and a neighbor in $R$. By \eqref{not-star-cutsets}, $w_0$ is not the center of a star cutset of $G$. By Lemma \ref{lem:short-pyramid}, $(H_1, w_0)$ is a short pyramid.  It follows that $w_0$ has exactly one neighbor $r_w$ in $R$ and $r_w \neq r_\ell$. This proves \eqref{w-in-R}. \\

\sta{\label{uniq-nbr-in-S} Let $\{a, b\} = \{v_0, w_0\}$ such that $s_t$ is the unique neighbor of $a$ in $S$. Then, $b$ has exactly one neighbor in $S$ and $b$ is non-adjacent to $s_t$.}

Let $H_2$ be the hole given by $H_2 = a \dd v \dd q_{k-1} \dd q_k \dd S \dd s_t \dd a$. Since $S$ contains neighbors of $b$, it follows that $b$ has at least three neighbors in $H_1$: $a$, $v$, and a neighbor in $S$. By \eqref{not-star-cutsets}, $b$ is not the center of a star cutset of $G$, and so by Lemma \ref{lem:short-pyramid}, $(H_2, w_0)$ is a short pyramid. It follows that $b$ has exactly one neighbor $s_w$ in $S$ and $s_w \neq s_t$. This proves \eqref{uniq-nbr-in-S}. \\

Suppose first that $s_t$ is the unique neighbor of $v_0$ in $S$. By \eqref{uniq-nbr-in-S}, it follows that $w_0$ has a unique neighbor $s_w$ in $S$ and $s_w \neq s_t$. Let $H_3$ be the hole given by $H_3 = v_0 \dd r_\ell \dd R \dd q_1 \dd Q \dd q_k \dd S \dd s_t \dd v_0$. It holds that $w_0$ has three pairwise non-adjacent neighbors $v_0, s_w, r_w$ in $H_3$, so $(H_3, w_0)$ is a proper wheel. But now by Lemma \ref{lem:proper-wheel}, $w_0$ is the center of a star cutset in $G$, contradicting \eqref{not-star-cutsets}.

Therefore, $s_t$ is the unique neighbor of $w_0$ in $S$. By \eqref{uniq-nbr-in-S}, it follows that $v_0$ has a unique neighbor $s_v$ in $S$ and $s_v \neq s_t$. Let $H_4$ be the hole given by $H_4 = v_0 \dd s_v \dd S \dd q_k \dd Q \dd q_1 \dd R \dd r_w \dd w_0 \dd v_0$. It follows that $(H_4, v)$ is a wheel and $v$ has $k$ neighbors in $H_4$. Next, let $H_5$ be the hole given by $H_5 = v_0 \dd s_v \dd S \dd q_k \dd Q \dd q_1 \dd R \dd r_\ell \dd v_0$. It follows that $(H_5, v)$ is a wheel and $v$ has $k-1$ neighbors in $H_5$. Since $k$ and $k-1$ have different parities, it follows that one of $(H_4, v)$ and $(H_5, v)$ is an even wheel, contradicting Lemma \ref{lem:no-even-wheel}. This completes the proof of the lemma. \end{proof}

Finally, we apply the previous lemma to the class of graphs with no even hole and no sector wheel.  Recall that a {\em sector wheel} is a wheel $(H, w)$ such that $N(w) \cap H$ is a path. 

\begin{thm}
\label{thm:star-cutset}
Let $G$ be minimal non-weakly FPE with no even hole and no sector wheel. Then, $G$ has no star cutset.  
\end{thm}
\begin{proof}
Assume for contradiction that $G$ has a star cutset. Let $v \in V(G)$ be such that there exists a cutset $X \subseteq N[v]$ of $G$ with $v \in X$. By Lemma \ref{lem:no-full-star-cutset}, $v$ is not the center of a full star cutset of $G$. This fact, together with  Lemma \ref{lem:v-not-full}, implies that there is exactly one component $C$ of $G \setminus N[v]$. Let $A$ be a connected component of  $G \setminus X$ such that 
$A$ is anticomplete to $C$. Then $A \subseteq N(v)$. Since $v$ is the center of a star cutset, it follows that $A \neq \emptyset$. Let $B = N(C) \cap N(A)$. Since $C$ is a connected component of $G \setminus N[v]$, it follows that $N(C) \subseteq N(v)$, and so $B \subseteq N(v)$. Also, note that $B$ can be empty. 
Suppose there exist $b_1, b_2 \in B$ such that $b_1$ is non-adjacent to $b_2$. Let $P_1$ be a path from $b_1$ to $b_2$ with $P_1^* \subseteq C$ and let $P_2$ be a path from $b_1$ to $b_2$ with $P_2^* \subseteq A$. Now, $P_1 \cup P_2$ is a hole and $v$ is complete to $P_2$ and anticomplete to $P_1^*$, so $(P_1 \cup P_2, v)$ is a sector wheel, a contradiction. Therefore, $B$ is a clique. Since $B = N(C) \cap N(A)$, it follows that $\{v\} \cup B$ separates $A$ from $C$, so $\{v\} \cup B$ is a clique cutset of $G$. But by Lemma \ref{lem:no-clique-cutset}, $G$ has no clique cutset, a contradiction. This completes the proof of the theorem. 



\end{proof}

\section{Putting it all together}
\label{sec:final}

In this section, we prove Theorem \ref{thm:main}.

\begin{proof}[Proof of Theorem \ref{thm:main}]
Let $G$ be an even-hole-free graph with no sector wheel, and suppose for a contradiction that $G$ is minimal non-weakly FPE. If $G$ has a clique cutset, then $G$ is weakly FPE by Lemma \ref{lem:no-clique-cutset}. If $G$ has a proper star cutset, then $G$ is weakly FPE by Theorem \ref{thm:star-cutset}. Therefore, $G$ has no star cutset. Note that $G$ is non-FPE and has no star cutset. Let $H$ be an induced subgraph of $G$ that is minimal with these properties, so in particular, $H$ has no star cutset, $H$ is non-FPE, and every induced subgraph of $H$ with no star cutset is FPE.  By Theorem \ref{thm:MNFPE}, since $H$ is minimal with no star cutset, $H$ is FPE, a contradiction. This completes the proof. 
\end{proof}
\bibliographystyle{acm}
\bibliography{sample}

\end{document}